\documentclass[11pt]{article}
\usepackage{amsmath, amsfonts, amsthm, amssymb, multicol}
\usepackage{graphicx}
\usepackage{float}
\usepackage{verbatim}% for comment

\allowdisplaybreaks

\usepackage[square,sort,comma,numbers]{natbib}
\usepackage{fancyhdr}
 
\numberwithin{equation}{section}

\newcommand{\ad}{\dim_\textup{A}}
\newcommand{\as}{\dim_\textup{A}^\theta}
\newcommand{\bd}{\dim_\textup{B}}
\newcommand{\ubd}{\overline{\dim}_\textup{B}}
\newcommand{\lbd}{\underline{\dim}_\textup{B}}
\newcommand{\qad}{\dim_\textup{qA}}
\newcommand{\hd}{\dim_\textup{H}}

\newtheorem{thm}{Theorem}[section]
\newtheorem{theorem}[thm]{Theorem}
\newtheorem{lemma}[thm]{Lemma}

\newtheorem{question}[thm]{Question}

\title{\vspace{-15mm}Interpolating between dimensions}

\author{Jonathan M. Fraser\\ \\
The University of St Andrews \\
\noindent  Email: jmf32@st-andrews.ac.uk}

	\date{}

\begin{document}

\maketitle

\begin{abstract}Dimension theory lies at the heart of fractal geometry and concerns the rigorous quantification of how large a subset of a metric space is.  There are many notions of dimension to consider, and part of the richness of the subject is in understanding how these different notions fit together, as well as how their subtle differences give rise to different behaviour.  Here we survey a new approach in dimension theory, which seeks to unify the study of individual dimensions by viewing them as different facets of the same object. For example, given two notions of dimension, one may be able to define a continuously parameterised family of dimensions which interpolates between them.  An understanding of this `interpolation function' therefore contains more information about a given object than the two dimensions considered in isolation.  We pay particular attention to two concrete examples of this, namely the \emph{Assouad spectrum}, which interpolates between the box and (quasi-)Assouad dimension, and the \emph{intermediate dimensions}, which interpolate between the Hausdorff and box dimensions.  
\end{abstract}

\section{Dimension theory and a new perspective}
\label{sec:1}

Roughly speaking, a \emph{fractal} is an object which exhibits complexity on arbitrarily small scales.  Such objects are hard to analyse, and cannot be easily measured.  Dimension theory is the study of how to measure fractals, specifically aimed at quantifying how they fill up space on small scales.  This is done by developing precise mathematical formulations of dimension and then developing techniques which can be used to compute these dimensions in specific settings, such as, for sets invariant under a dynamical system or generated by a random process, see Figure \ref{fractfig}.  There are many ways to define dimension which naturally extend our intuitive idea that lines have dimension 1 and  squares have dimension 2, etc.  The box dimension is a particularly natural and easily digested notion of dimension, which comes from understanding how a coarse measure of size behaves  as the resolution increases.  More precisely, given a bounded set $F \subseteq \mathbb{R}^d$ and  a  scale (resolution) $r>0$,  let $N_r(F)$ denote the minimum number of sets of diameter $r$ required to cover $F$, see Figure \ref{fig:bd}.  This should increase as $r \to 0$ and it is natural to expect $r \approx r^{-\delta}$ for some $\delta>0$, which can be readily interpreted  as the `dimension' of $F$.  As such, the \emph{upper box dimension} of $F$ is defined by
\[
\ubd  F \ = \  \limsup_{r \to 0} \frac{\log N_r(F)}{-\log r}.
\]
If   the $\limsup$ is replaced by  $\liminf$,  one gets the \emph{lower box dimension} $\lbd F$.  However, often the $\limsup$ and $\liminf$ agree, in which case we refer to the common value as the \emph{box dimension}, denoted by  $\bd F$.   Despite how convenient and natural this definition is, it has some theoretical disadvantages, such as not being countably stable, see \cite[page 40]{falconer}.  A more sophisticated notion, which is similar in spirit, is the \emph{Hausdorff dimension}.  This can be defined, for any set $F \subseteq \mathbb{R}^d$, by
\begin{eqnarray*}
\hd  F \ = \  \inf \Bigg\{ \  \alpha>0 &:&  \text{for all $\varepsilon>0$ there exists a cover $\{U_i\}$ of $F$} \\
&\,& \hspace{28mm}  \text{ such that  $\sum_i |U_i|^\alpha < \varepsilon$ } \Bigg\}.
\end{eqnarray*}
The key difference here is that sets with vastly different diameters  are permitted in the covers and their contribution to the `dimension' is weighted according to their diameter, see Figure \ref{fig:hd}.  In particular, it is easily seen that the Hausdorff dimension is countably stable.  Both the Hausdorff and box dimension measure the size of the whole set, giving rise to an ``average dimension''. It is often the case that more extremal information is required, for example in embedding theory, see \cite{robinson}.  The \emph{Assouad dimension} is designed to capture this information and is defined, for any set $F \subseteq \mathbb{R}^d$, by
\begin{eqnarray*}
\dim_\text{A} F \ = \  \inf \Bigg\{ \  \alpha>0 &:& \text{     there exists a constant $C >0$ such that,} \\
&\,& \hspace{10mm}  \text{for all $0<r<R $ and $x \in F$ we have } \\ 
&\,&\hspace{20mm}  \text{$ N_r\big( B(x,R) \cap F \big) \ \leq \ C \bigg(\frac{R}{r}\bigg)^\alpha$ } \Bigg\}.
\end{eqnarray*}
The key point here is that one does not seek covers of the whole space, but only a small ball, and  the expected covering number  is appropriately normalised, see Figure \ref{fig:ad}.   One of the joys of dimension theory is in understanding how these different notions of dimension relate to each other and how they behave in different settings.  It is a simple exercise to demonstrate that
\[
\hd F \leq  \lbd F \leq \ubd F \leq \ad F
\]
for any bounded $F \subseteq \mathbb{R}^d$, and that these inequalities can be strict inequalities or equalities in any combination. Equality throughout can be interpreted as a manifestation of  `strong homogeneity'. For example, if $F$ is Ahlfors-David regular then $\hd F = \bd F = \ad F$.  

There are of course many other notions of dimension, each important in its own right and motivated by particular questions or applications.  We omit discussion of these, but other examples include the packing, lower, quasi-Assouad, modified box, topological, Fourier, among many others.  We refer the reader to \cite{bishopperes, techniques, falconer, mattila, robinson} for more background on dimension theory, including a  thorough investigation of the  basic properties of the various notions of dimension.

\begin{figure}[H]
\centering
		\includegraphics[width=46mm]{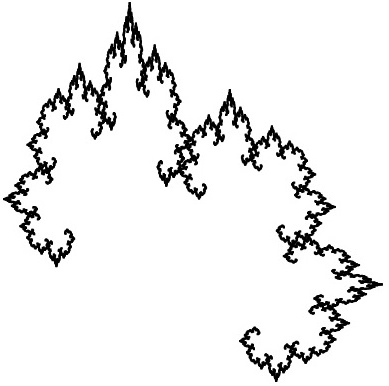}
\includegraphics[width=44mm]{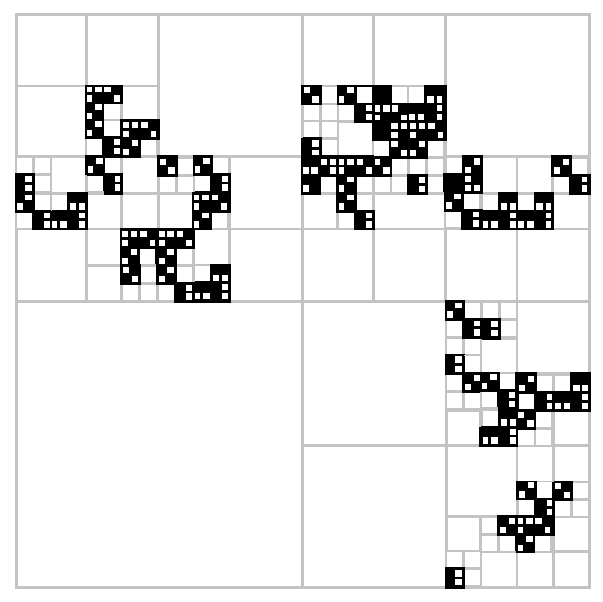}
\includegraphics[width=41mm]{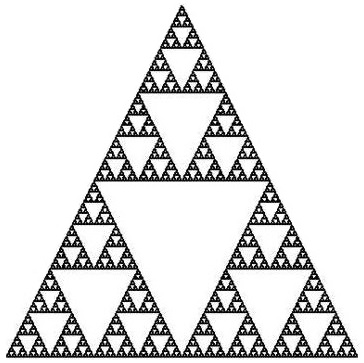}
\caption{Three fractals: a self-affine set  (left), a random set generated by \emph{Mandelbrot percolation} (centre), and the self-similar \emph{Sierpi\'nski triangle}   (right).}
\label{fractfig}
\end{figure}

\begin{figure}[H]
\includegraphics[width=50mm]{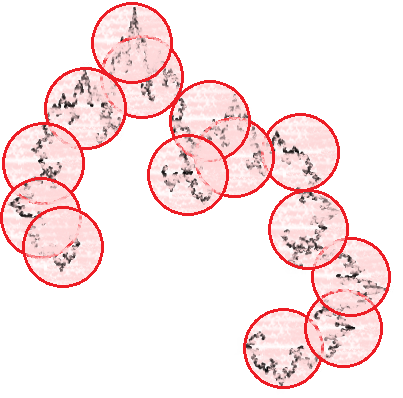}
\includegraphics[width=50mm]{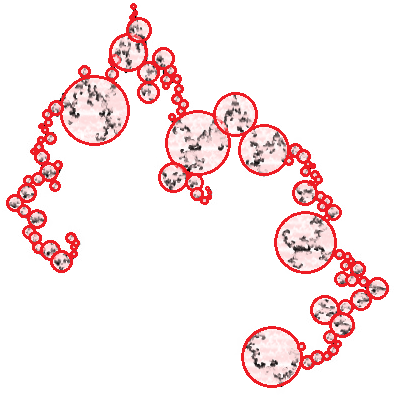}
\includegraphics[width=50mm]{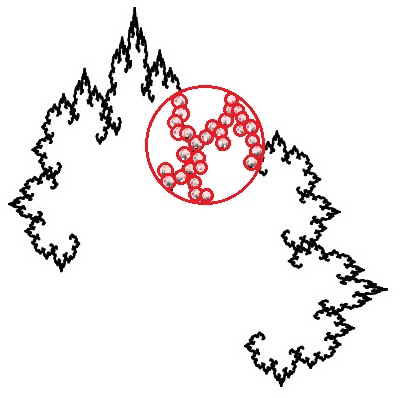}
\caption{Left: an efficient covering of  the self-affine set from Figure \ref{fractfig} by balls of the same radius.  Counting the number of balls required for such a cover as the radius tends to 0 gives rise to the \emph{box dimension}. Centre: an efficient covering of  the same set by balls of arbitrarily varying radii.  Understanding the weighted sum of diameters of the sets in such a cover gives rise to the \emph{Hausdorff dimension}. Right: an efficient covering of  the same set by smaller  balls of the same radius.  Counting the number of balls required for such a cover, optimised over all larger balls and all pairs of scales, gives   rise to the \emph{Assouad dimension}.}
\label{fig:bd}     \label{fig:hd}  \label{fig:ad} 
\end{figure}

The main purpose of this article is to motivate a new perspective in dimension theory.  Rather than view these notions of dimension in isolation, we should try to view them as different facets of the same object. This approach will  give rise to a continuum of dimensions, which fully describes the scaling structure of the space, both locally and globally.  Moreover, this will yield a more nuanced understanding of the individual notions of dimensions as well as insight into the somewhat philosophical question of how to define dimension itself.  This sounds rather grand and  ambitious, but by focusing our attention slightly and applying this philosophy in particular settings, an interesting and workable theory has started to emerge.   More concretely, given dimensions $\dim$ and $\textup{Dim}$ which generally satisfy $\dim F \leq \textup{Dim} \, F$, we wish to introduce  a parameterised  family of dimensions $d_\theta$, with parameter $\theta \in [0,1]$, which (ideally) satisfies:  \newpage
\begin{itemize}
\item $d_0  = \dim$
\item $d_1  = \textup{Dim} $
\item $\dim F \leq d_\theta(F) \leq \textup{Dim} \, F, $ for all $ \theta \in (0,1)$  and all reasonable sets $F$
\item for a given $F$, $d_\theta(F)$ varies  continuously in $\theta$.
  \end{itemize}
Moreover,
\begin{itemize}
\item the definition of $d_\theta$ should be natural, sharing the philosophies of both $\dim$ and $ \textup{Dim}$
\item $d_\theta$ should give rise to a rich and workable theory.
\end{itemize}
The most important of these points are the final two.  One can achieve the first four in any number of trivial and meaningless ways, but the key idea is that the function $\theta \mapsto d_\theta(F)$ should be ripe with easily interpreted, meaningful, and nuanced information regarding the set $F$.    If this can be achieved then the rewards are likely to include:
\begin{itemize}
\item a better understanding of  $\dim$ and $\textup{Dim}$   
\item an explanation of one type of behaviour changing into another   
\item more information, leading to better applications   
\item a (large) new set of questions  
\item fun.
\end{itemize}
In the following subsections we describe two concrete examples of this philosophy in action.

\subsection{The Assouad spectrum}
\label{assspec}

The \emph{Assouad spectrum}, introduced by Fraser and Yu in 2016 \cite{Spectraa}, aims to interpolate between the upper box dimension and the Assouad dimension.  The parameter  $\theta \in (0,1)$ serves to fix the relationship between the  two scales $r<R$ used to define the Assouad dimension, by setting $R=r^{\theta}$.  As such, the Assouad spectrum  of $F \subseteq \mathbb{R}^d$ is defined by
\begin{eqnarray*}
\as F \ = \  \inf \Bigg\{ \  \alpha>0 &:& \text{     there exists a constant $C >0$ such that,} \\
&\,& \hspace{10mm}  \text{for all $0<r<1 $ and $x \in F$ we have } \\ 
&\,&\hspace{20mm}  \text{$ N_{r}\big( B\big(x,r^{\theta}\big) \cap F \big) \ \leq \ C \bigg(\frac{r^{\theta}}{r}\bigg)^\alpha$ } \Bigg\}.
\end{eqnarray*}
At this point it might seem equally natural to bound the two scales away from each other by considering all $0<r \leq  R^{1/\theta}$ rather than fixing $r=R^{1/\theta}$.  Rather than go into details here, we simply observe that fixing the relationship between the scales is both easier to work with and provides strictly more information than the alternative, see \cite{canadian}. We also note that  in \cite{Spectraa} the scales were denoted by $R^{1/\theta}$ and $R$, rather than $r$ and $r^\theta$. These two formulations are clearly equivalent but the notation we use here seems a little less cumbersome, however, in certain situations it is more natural  to use $R^{1/\theta}$ and $R$.   It was established  in \cite{Spectraa} that $\as F $ is:
\begin{itemize}
\item continuous in $\theta \in (0,1)$, see \cite[Corollary 3.5]{Spectraa}   
\item Lipschitz on any closed subinterval of $(0,1)$, see \cite[Corollary 3.5]{Spectraa}  
\item not necessarily monotonic   (but often is), see \cite[Proposition 3.7 and Section 8]{Spectraa}.
\end{itemize}
Moreover,  we have the following general bounds, adapted from \cite[Proposition 3.1]{Spectraa}.
\begin{lemma} \label{standardbounds}
For any bounded set $F \subseteq \mathbb{R}^d$, 
\[
\ubd F \leq \as F \leq \min \left\{ \frac{\ubd F}{1-\theta}, \ \ad F \right\} .
\]
\end{lemma}

\begin{proof}
Let $s> \ubd F$, $x \in F$ and $r \in (0,1)$.  By definition there exists $C>0$ depending only on $s$ such that
\[
N_{r}\big(B\big(x,r^\theta\big)\cap F\big)\leq N_{r}(F ) \leq C  r^{-s} = C \left( \frac{r^\theta}{r} \right)^{s/ (1-\theta)}
\]
which implies  $\as F \leq s/(1-\theta)$ and since $s> \ubd F$ was arbitrary, the upper bound follows, noting that $\as F \leq \ad F$ is trivial.  

For the lower bound, we may assume $\ubd F>0$ and let $0 < t < \ubd F<s$.  Covering $F$ with $r^\theta$-balls and then covering each of these $r^\theta$-balls with $r$-balls, we obtain
\[
N_{r}(F) \leq N_{r^\theta}(F) \,  \left( \sup_{x\in F}N_{r}\big(B\big(x,r^\theta\big)\cap F\big) \right).
\]
Again,  by definition, there exist  arbitrarily small $r>0$ such that
\[
\sup_{x\in F}N_{r}\big(B\big(x,r^\theta\big)\cap F\big) \geq  \frac{N_{r}(F)}{N_{r^\theta}(F) } \geq \frac{r^{-t}}{r^{-s\theta} } =  \left( \frac{r^\theta}{r} \right)^{\frac{s\theta-t}{\theta-1}}
\]
which establishes $\as F \geq \frac{t-s\theta}{1-\theta}$ and, since $s$ and $t$ can be made arbitrarily close to $\ubd F$, the lower bound follows. 
\end{proof}

A useful consequence of Lemma \ref{standardbounds} is that  $\as F \to \ubd F$ as $\theta \to 0$ for any bounded $F$.  However,    $\as F$ may \emph{not} approach $\ad F$ as $\theta \to 1$.  In fact, it was proved in \cite{canadian} that  $\as F \to \qad F$ as $\theta \to 1$, where $\qad F$ is the \emph{quasi}-Assouad dimension.  In many cases the quasi-Assouad dimension and Assouad dimension coincide and so the intended interpolation is achieved. Moreover, the appearance of Assouad dimension in Lemma \ref{standardbounds} may be replaced by the quasi-Assouad dimension.

  Generally, one has $\qad F \leq \ad F$ and if this inequality is strict, then the intended interpolation is not achieved.  However, an approach for ``recovering'' the interpolation was outlined in \cite{Spectraa}.   Let $\phi : [0,1] \to [0,1]$ be an increasing continuous function such that $\phi(R) \leq R$ for all $R \in [0,1]$.  The $\phi$-\emph{Assouad dimension}, introduced in \cite{Spectraa}, is defined by \newpage
\begin{eqnarray*}
\ad^\phi F \ = \  \inf \Bigg\{ \  \alpha>0 &:& \text{     there exists a constant $C >0$ such that,} \\
&\,& \hspace{2mm}  \text{for all $0<r\leq \phi(R) \leq R \leq 1 $ and $x \in F$ we have } \\ 
&\,&\hspace{29mm}  \text{$ N_{r}\big( B(x,R) \cap F \big) \ \leq \ C \bigg(\frac{R}{r}\bigg)^\alpha$ } \Bigg\}.
\end{eqnarray*}
The goal is now to identify  precise conditions on $\phi$ which guarantee  $\ad^\phi F = \ad F $. Resolution of this problem for a particular $F$ gives precise information on how the Assouad dimension of $F$ can be \emph{witnessed} and, moreover, completes the interpolation between the upper box and Assouad dimension in a precise sense.  Often $\as F = \ad F$ for some $\theta \in (0,1)$, in which case the threshold for witnessing the Assoaud dimension is provided by the function $\phi(R) = R^{1/\theta}$.    The $\phi$-\emph{Assouad dimension} has been considered in detail by Garc\'ia, Hare, and Mendivil  \cite{Hare2, Hare3} and Troscheit  \cite{ sascha}.

Various other dimension spectra are introduced in \cite{Spectraa}, including the \emph{lower spectrum}, which is the natural dual to the Assouad spectrum and lives in between the lower dimension and the lower box dimension.  This has been investigated, in conjunction with the Assouad spectrum, by Chen, Wu and Chang \cite{chen1, chen2}, Hare and Troscheit \cite{Hare4} and Fraser and Yu \cite{Spectrab}.

\subsection{Intermediate dimensions}
\label{intdim}

The \emph{intermediate dimensions}, introduced by Falconer, Fraser and Kempton in 2018 \cite{FalconerFraserKempton},  aim to interpolate between the Hausdorff and  box dimensions.  The parameter  $\theta \in (0,1)$ serves to restrict the discrepancy between the size of covering sets in the definition of the Hausdorff dimension by insisting that $|U_i| \leq |U_j|^\theta$ for all $i,j$.  As such, the $\theta$-intermediate dimensions of a bounded set $F \subseteq \mathbb{R}^d$ are defined by
\begin{eqnarray*}
\dim_\theta  F \ = \  \inf \Bigg\{ \  \alpha>0 &:&  \text{for all $\varepsilon>0$ there exists a cover $\{U_i\}$ of $F$ } \\
&\,& \hspace{0mm}  \text{with $|U_i| \leq |U_j|^\theta$ for all $i,j$  such that  $\sum_i |U_i|^\alpha < \varepsilon$ } \Bigg\}.
\end{eqnarray*}
In fact, \cite{FalconerFraserKempton} considers upper and lower intermediate dimensions, but we restrict our attention here to the lower version. It was proved in \cite{FalconerFraserKempton} that    $\dim_\theta  F$  is:
\begin{itemize}
\item   continuous in $\theta \in (0,1)$, see    \cite[Proposition 2.1]{FalconerFraserKempton}
\item monotonically increasing  
\item  bounded between the Hausdorff and lower box dimension, that is, for bounded $F$
\[
\hd F \leq \dim_\theta  F \leq  \lbd F  
\] \vspace{-7mm}
\item and satisfies appropriate versions of the \emph{mass distribution principle} and \emph{Frostman's lemma}, see  \cite[Propositions 2.2-2.3]{FalconerFraserKempton}.
\end{itemize}
Next we establish general lower bounds for the intermediate dimensions which involve the   Assouad dimension, see  \cite[Proposition 2.4]{FalconerFraserKempton}. In the proof we rely on the following mass distribution principle, first proved in \cite[Proposition 2.2]{FalconerFraserKempton}. The main difference between Lemma \ref{mdp} and the usual mass distribution principle, see \cite[4.2]{falconer}, is that a family of measures $\{\mu_r\}$ is used instead of a single measure.

\begin{lemma}\label{mdp}
Let $F$ be a Borel subset of $\mathbb{R}^d$,   $0\leq \theta \leq 1$ and $s\geq 0$. Suppose that there are numbers $a, c, r_0 >0$ such that for all $0< r\leq r_0$  we can find a Borel measure $\mu_r$ supported by $F$ with
$\mu_r (F) \geq a $, such that 
\begin{equation}\label{mdiscond}
\mu_r (U) \leq c|U|^s
\end{equation}
for all Borel sets  $U \subseteq \mathbb{R}^d $ with $r \leq |U|\leq r^\theta$. Then  $\dim_\theta F \geq s$. 
 \end{lemma}

\begin{proof}
Let $\{U_i\}$ be a cover of $F$ such that $r \leq |U_i|\leq r^\theta$ for all $i$ and some $r\leq r_0$. We may clearly assume the $U_i$ are Borel (even closed). Then
\[
a\ \leq \ \mu_r(F) \ \leq\  \mu_r\Big(\bigcup_i U_i\Big)\ \leq \ \sum_i \mu_r(U_i) \ \leq \  c\sum_i |U_i|^s,
\]
so that $\sum_i |U_i|^s \geq a/c>0$ for every admissible cover (by sets with sufficiently small diameters) and therefore $\dim_\theta F \geq s$. 
\end{proof}

\begin{lemma} \label{generalboundaa}
For bounded $F \subseteq \mathbb{R}^d$ and $\theta \in (0,1)$, we have
\[
\dim_\theta  F \geq \ad  F - \frac{\ad F -  \lbd F}{\theta}.
\]
\end{lemma}

\begin{proof}
Fix $\theta \in (0,1)$ and assume that $\underline{\dim}_\textup{B} F >0$, since otherwise there is nothing to prove.  Let
\[
0<s< \underline{\dim}_\textup{B} F  \leq  \dim_\textup{A} F < t < \infty
\]
and $r \in (0,1)$ be given.  Since $s< \underline{\dim}_\textup{B} F$, there exists a  constant $C_0$ such that there is an $r$-separated set of points in $F$ of cardinality at least $C_0 r^{-s}$.  Let $\mu_r$ be a uniformly distributed probability measure supported on this set of points. 

Let $U \subseteq \mathbb{R}^d$ be a Borel set with $|U| = r^\gamma$ for some $\gamma \in [\theta, 1]$.  Since $\dim_\textup{A} F < t$,  there exists a  constant $C_1$  such that $U$ intersects at most $C_1 ( r^\gamma/r)^t$ points in the support of $\mu_r$.  Therefore
\[
\mu_r (U) \leq C_1r^{(\gamma-1)t} C_0^{-1} r^{s} = C_1C_0^{-1} \ |U|^{(\gamma t - t+ s)/ \gamma} \leq  C_1C_0^{-1}  \ |U|^{(\theta t - t+ s)/ \theta},
\]
which, using Lemma \ref{mdp}, implies that
\[
\dim_\theta F \geq (\theta t - t+ s)/ \theta = t - \frac{t-s}{\theta}.
\]
Letting $t \to  \ad F$ and $s \to \underline{\dim}_\textup{B} F$ yields the desired result. 
\end{proof}

It follows from this lemma that $\dim_\theta  F \to \lbd F$ as $\theta \to 1$.  In contrast, it was shown in \cite{FalconerFraserKempton} that $\dim_\theta  F$ does not necessarily approach $\hd F$ as $\theta \to 0$. Indeed, the above lemma shows that if  $\lbd F = \dim_\textup{A} F$, then $\dim_\theta  F  = \lbd F  = \dim_\textup{A} F$ for all $\theta \in (0,1)$.

Lemma \ref{generalboundaa} should be compared with Lemma \ref{standardbounds}.  For example, combining the two results, one sees that if $F \subseteq \mathbb{R}^d$ and either    $\bd F = 0$ or $\bd F  = d$, then both the intermediate dimensions and Assouad spectrum are constant (and equal to $\bd F$).

\section{Examples}
\label{sec:2}

\subsection{Countable sets}
\label{countable}

Fix $p>0$, and let $F_p = \{ n^{-p} \ : \ n \in \mathbb{N}\}$.   It is straightforward to show that
\[
\hd F_p = 0 <  \bd F_p = \frac{1}{1+p} <  \ad F_p = 1.
\]
Moreover, it was shown in \cite[Corollary 6.4]{Spectraa} that
\[
 \as F_p = \min\left\{ \frac{1}{(1+p)(1-\theta)} , \ 1\right\}  
\]
and in \cite[Proposition 3.1]{FalconerFraserKempton}  that
\[
\dim_\theta  F_p = \frac{\theta}{\theta+p},
\]
see Figure \ref{fig:countable}. Therefore these simple examples provide a clear exposition of dimension interpolation in action, noting that genuine continuous interpolation between the dimensions considered is achieved in each case.

\begin{figure}[H]
\centering
		\includegraphics[width=\textwidth]{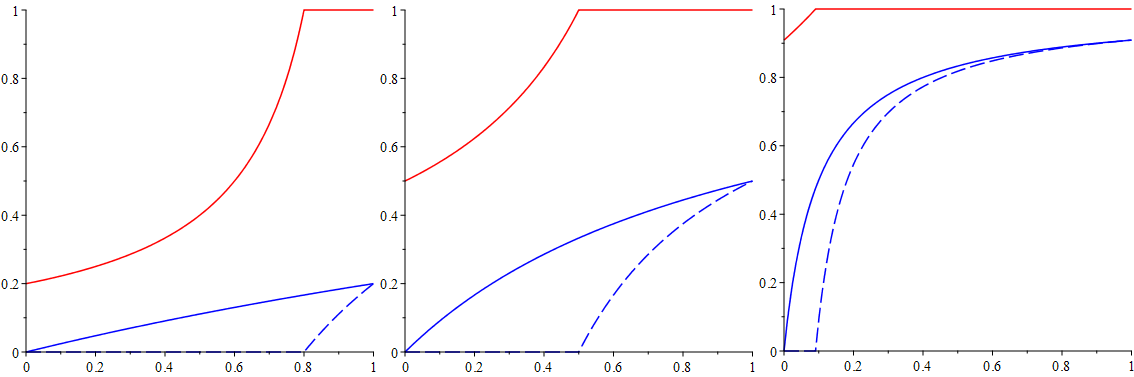}
\caption{Plots of $ \as F_p$ (red) and $\dim_\theta  F_p$ (solid blue) as functions of $\theta$ for different values of $p$.  On the left, $p=4$, in the centre $p=1$, and on the right $p=1/10$.  For reference, the general lower bounds   from Lemma \ref{generalboundaa} for the intermediate dimensions are shown as a dashed blue line.  The general upper bounds from Lemma \ref{standardbounds} for the Assouad spectrum are achieved. }
\label{fig:countable}  
\end{figure}

\subsection{Self-affine sets}
\label{affine}

One of the most natural and important families of set which exhibit distinct Hausdorff, box and Assouad dimensions are the self-affine carpets introduced by Bedford and McMullen \cite{bedford, mcmullen}.  These sets are constructed as follows. Divide the unit square $[0,1]^2$ into an $m \times n$ grid, for integers  $n>m \geq 2$,  and select a collection of $N \geq 2$  rectangles formed by the grid.  Label the rectangles $1, \dots, N$ and, for each rectangle $i$, let $S_i$ denote the  affine map which maps $[0,1]^2$ onto $i$ by first applying the map $(x,y) \mapsto (x/m, y/n)$ and then translating.  The \emph{Bedford-McMullen carpet}   is defined to be the unique non-empty compact set $F$ satisfying
\[
F = \bigcup_{i=1}^N S_i(F),
\]
see Figure \ref{fig:carpetspec}. The fact that this formula defines such a set uniquely is a well-known result in fractal geometry concerning \emph{iterated function systems}, see   \cite[Chapter 9]{falconer} for the details. 

 In order to state known dimension  formulae for $F$, let $M \in [1,m]$ denote the number of distinct columns in the grid containing chosen rectangles $i$,  $C_j \in [1,n]$ denote the number of chosen rectangles in the  $j$th nonempty column for $j\in \{1, \dots, M\}$, and $C_{\max} = \max_j C_j$.  Bedford and McMullen independently computed the box and Hausdorff dimensions of $F$ in 1984   \cite{bedford, mcmullen} and the Assouad dimension was computed by Mackay in 2011 \cite{mackay}.  The respective formulae are
\[
\hd  F \ = \ \frac{\log \sum_{j} C_j^{\log m/ \log n}}{\log m},  
\]
\[
\bd F \ = \    \frac{\log M}{\log m} \, + \, \frac{\log ( N/M)}{\log n},   
\]
and
\[
\ad  F \ = \  \frac{\log M}{\log m} \, + \,  \frac{\log C_{\max}}{\log n} .  
\]
Note that if $C_j<C_{\max}$ for some $j$, then the Hausdorff, box and Assouad dimensions are all distinct.  This is called the \emph{non-uniform fibres} case and is the case of interest.  In fact, in the \emph{uniform} fibres case, the three dimensions coincide.  Therefore, from now on we restrict our attention to the non-uniform fibres setting, where computation of the Assouad spectrum and intermediate dimensions is relevant.  It was recently proved in \cite[Corollary 3.5]{Spectrab} that, for $\theta \in (0, \log m/ \log n]$,
\[
\as F \ = \    \frac{\log M- \theta \log ( N/C_{\max})}{(1-\theta) \log m} \, + \, \frac{\log ( N/M) - \theta \log C_{\max}}{(1-\theta) \log n}  
\]
and for $\theta \in [ \log m/ \log n, 1)$
\[
\as F = \ad F,
\]
see Figure \ref{fig:carpetspec}. In particular, a single phase transition occurs at $\theta = \log m /\log n$, and a short calculation reveals that this is  strictly greater  than
\[
1- \frac{\ubd F}{\ad F}
\]
which is where the single phase transition occurs in the general upper bound from Lemma \ref{standardbounds}.   Therefore, the general upper bound for $\as F$ is never achieved by a Bedford-McMullen carpet in the non-uniform fibres setting.

The intermediate dimensions of $F$ were considered in \cite{FalconerFraserKempton}, where it was established that $\dim_\theta F \to \hd F$ as $\theta \to 0$.  Recall that this `genuine interpolation' is not satisfied for all sets.  A precise formula for $\dim_\theta F$ currently seems out of reach, but the following bounds were established in \cite[Propositions 4.1 and 4.3]{FalconerFraserKempton}, see Figure \ref{fig:carpetspec}.  For $0< \theta < \left( \frac{\log m }{2\log n}\right)^2 $ we have the upper bound
\[
\dim_\theta F \leq \hd F +\frac{2 (\log C_{\max}) \log\left( \frac{\log n }{\log m}\right) }{-(\log n) \log \theta},
\]
which importantly establishes $\dim_\theta F \to \hd F$ as $\theta \to 0$, but only improves on the trivial bound of $\dim_\theta F  \leq \bd F$ for very small values of $\theta$.  For example, for the carpet considered in Figure \ref{fig:carpetspec} this improvement is only achieved for $\theta$  smaller than around $10^{-13}$.  Also, for all $\theta \in (0,1)$ we have the lower bound
\[
\dim_\theta F \geq  \hd F + \theta \frac{\log N - h}{\log m},
\]
where
\[
h = -m^{-\hd F} \sum_j C_j^{\log m/ \log n} \left(\left(\frac{\log m }{\log n}-1\right) \log C_j -  \hd F \log m\right) 
\]
is the entropy of the McMullen measure. A short calculation shows that $0 < h \leq \log N$ with $h = \log N$ if and only if $F$ has uniform fibres.  Therefore, in the non-uniform fibres case we have $\hd F< \dim_\theta F$ for all $\theta \in (0,1)$.  This lower bound improves on the general lower bound from Lemma \ref{generalboundaa} for the carpet considered in Figure \ref{fig:carpetspec} for $\theta \leq 0.96$. In the absence of a precise formula, we  ask the following questions.
\begin{question}
For $F$ a Bedford-McMullen carpet with non-uniform fibres, is it true that $ \dim_\theta F < \bd F$ for all $\theta \in (0,1)$?  Moreover, is it true that $\dim_\theta F$ is strictly increasing, differentiable, or analytic?  
\end{question}

\begin{figure}[H]
\centering
		\includegraphics[width=\textwidth]{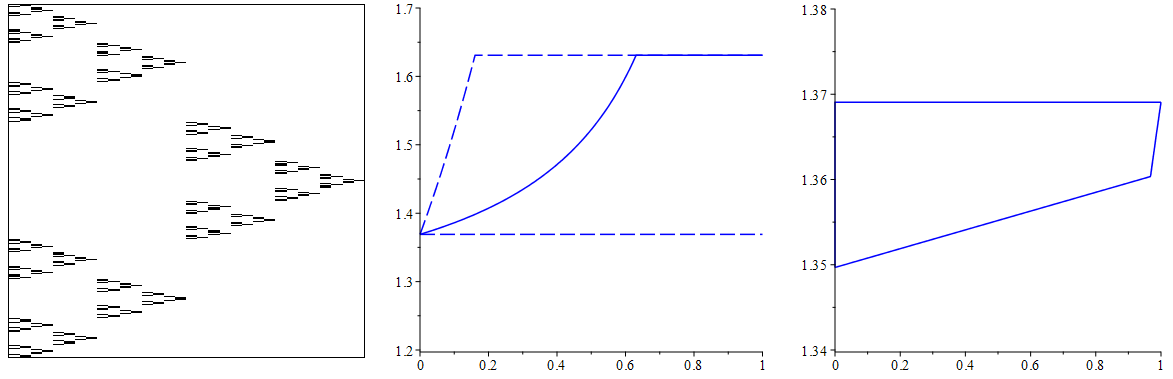}
\caption{Left:  an example of a self-affine carpet $F$ where $n=3$, $m=2$, $N=3$, $M=2$, $C_1=2$, $C_2=1$ and $C_{\max}=2$.  Centre: a plot of $ \as F$  (solid blue) as a function of $\theta$.  For reference, the general upper and lower bounds for the Assouad spectrum from Lemma \ref{standardbounds} are shown as dashed blue lines. Right: plots of the upper and lower bounds for $\dim_\theta F$.  The phase transition in the lower bound comes from the bounds established in \cite{FalconerFraserKempton} switching with the general bounds from Lemma \ref{generalboundaa}.}
\label{fig:carpetspec}  
\end{figure}

\subsection{Self-similar sets and random sets}

The examples discussed so far (the countable sets, and self-affine carpets with non-uniform fibres) are particularly well-suited to the models of interpolation  we discuss in this article.  In particular, the Hausdorff, box, and Assouad dimensions are all distinct, and the intermediate dimensions and Assouad spectrum achieve genuine interpolation between these three dimensions.  Recall that this is not always the case.  Here we discuss two natural families of sets, for which the desired interpolation is not achieved: self-similar sets with overlaps, and Mandelbrot percolation.

We restrict our attention to self-similar sets in $\mathbb{R}$, but interesting questions remain open in higher dimensions.  Let $\{S_i\}_i$ be a finite collection of contracting orientation preserving  similarities  mapping $[0,1]$ into itself.  That is, for each $i$, there are constants $c_i \in (0,1)$ and $t_i \in [0,1-c_i]$ such that $S_i$ is given by $S_i(x) = c_i x +t_i$.  Similar to the setting of self-affine carpets, there exists a unique non-empty compact set $F \subseteq [0,1]$ satisfying
\[
F = \bigcup_i S_i(F).
\]
Such sets $F$ are known as \emph{self-similar}, see   \cite[Chapter 9]{falconer}.  It is  well-known that if there exists a non-empty open set $U \subseteq [0,1]$ such that   $\cup_iS_i(U) \subset U$ and the sets $S_i(U)$ are pairwise disjoint, then
\[
\hd F = \bd F = \ad F = s
\]
where $s \in (0,1]$ is the unique solution to \emph{Hutchinson's formula} $\sum_i c_i^s = 1$.  In particular, this `separation condition', known as the \emph{open set condition} (OSC), guarantees that the pieces $S_i(F)$ do not overlap too much and thus the images of $F$ under  iterates of the defining maps directly  give rise to efficient covers of $F$, facilitating calculation of dimension.  It also guarantees sufficient homogeneity to ensure equality of the three dimensions we discuss.  In particular, self-similar sets satisfying the OSC are not interesting from our dimension interpolation perspective.  However, if the OSC fails, then the Assouad dimension can strictly exceed the box dimension, see \cite{fraserassouad, fraserrob}.  On the other hand, the Hausdorff and box dimension always coincide for self-similar sets, see \cite[Corollary 3.3]{techniques}.  Thus, the natural object to consider here is the Assouad spectrum.  The following  result was proved in \cite[Corollary 4.2]{Spectrab}.

\begin{theorem}
  Let $F \subseteq \mathbb{R}$ be a self-similar set which does not have `super-exponential concentration of cylinders'.  Then for all $\theta \in (0,1)$
\[
\as F = \bd F.
\]
\end{theorem}

In particular, this result implies that genuine interpolation between the box dimension and the Assouad dimension is \emph{not} achieved for these self-similar sets whenever the  Assouad dimension strictly exceeds its box dimension. It remains open whether the conclusion of the above result is true for \emph{all} self-similar sets.  This theorem was proved using a recent result  of Shmerkin \cite{pablo} and we refer the reader to this paper for more details on the `super-exponential concentration' assumption.  We note, however, that this assumption is satisfied if  the  semigroup generated by the defining maps is free (that is, there are no `exact overlaps') and the parameters $t_i$ and $c_i$ defining the maps are algebraic.    

Mandelbrot percolation is a natural random process giving rise to fractals which are \emph{statistically} self-similar, see \cite[Section 15.2]{falconer}.  We begin with the unit cube $M_0 = [0,1]^d$, a fixed integer $m \geq 2$, and a probability $p \in (0,1)$.  At the first step of the construction we divide $M_0$ into $m^d$ (closed) cubes of side length $m^{-1}$ and for each cube we independently choose to `keep it' with probability $p$, or `throw it away' with probability $(1-p)$.  We let $M_1$ be the collection of kept cubes and we then repeat this process inside each kept cube independently, denoting the collection of kept cubes at stage $n$ by $M_n$.  The limit set is then defined by $M = \cap_n M_n$, see Figure \ref{fractfig} for an example with $d=m=2$.  It is well-known that if $p> m^{-d}$, then $M$ is non-empty with positive probability.  Moreover, if we condition on  $M$ being non-empty, then
\[
\hd M = \bd M = d + \frac{\log p }{\log m}  \in (0,d)
\]
almost surely.  It was shown in \cite[Theorem 5.1]{frasermiaotro} that, conditioned on $M$ being non-empty,
\begin{equation} \label{aaa1}
\ad M = d
\end{equation}
almost surely, and therefore it is natural to consider the Assouad spectrum of $M$.  However, it was proved in \cite{Spectrab, Troscheit18, hanphd} that,   conditioned on $M$ being non-empty, almost surely
\begin{equation} \label{aaa}
\as M = \bd M
\end{equation}
for all $\theta \in (0,1)$.  Therefore, again we see that genuine interpolation between the box dimension and Assouad dimension is \emph{not} achieved by the Assouad spectrum for Mandelbrot percolation. However, using the finer analysis introduced in \cite{Spectraa} and discussed in Section \ref{assspec}, it is possible to observe the interpolation by considering $\dim_A^\phi M$ for different functions $\phi$.    Troscheit proved the following dichotomy in \cite{sascha}.

\begin{theorem}
If 
\[
\frac{\log(R/\phi(R))}{\log|\log R|} \to 0
\]
as $R \to 0$, then, conditioned on $M$ being non-empty, almost surely
\[
\dim_\textup{A}^\phi M = d = \ad M.
\]
Moreover, if
\[
\frac{\log(R/\phi(R))}{\log|\log R|} \to \infty
\]
as $R \to 0$, then, conditioned on $M$ being non-empty, almost surely
\[
\dim_\textup{A}^\phi M = \bd M = d + \frac{\log p }{\log m}.
\]
\end{theorem}

Note that this result implies \eqref{aaa} by considering $\phi(R) = R^{1/\theta}$ and \eqref{aaa1} by considering $\phi(R) = R$.  A similar dichotomy, with the same threshold on $\phi$, was  obtained in a different random setting in \cite{Hare3}. The Assouad spectrum of random self-affine carpets was considered in \cite{frasersascha}.

\section{Applications: bi-Lipschitz and bi-H\"older distortion}
\label{sec:3}

A key aspect of this new perspective in dimension theory is in its applications.  The idea is that  if we can interpolate between two given dimensions in a meaningful way, then we will get strictly better information than when the dimensions are considered in isolation. This better information should, in turn, yield stronger applications. 

A common application of dimension theory is derived from the fact that dimensions are often invariant, or approximately  invariant in a quantifiable sense, under a family of  transformations.  For example, the Hausdorff, box and Assouad dimensions are all invariant under bi-Lipschitz maps and therefore provide useful invariants in the problem of classification up to bi-Lipschitz image.  The Assouad spectrum and intermediate dimensions are also invariant under  bi-Lipschitz maps and therefore provide a continuum of invariants in the same context.  Recall that an injective  map $f:X \to \mathbb{R}^d$ is \emph{bi-Lipschitz} if there exists a constant $C\geq 1$ such that for all distinct   $x,y \in X$
\begin{equation} \label{bilipz}
C^{-1} |x-y| \leq |f(x)-f(y)| \leq C |x-y|.
\end{equation}
Here we assume that $X$ is a bounded subset of $\mathbb{R}^d$. In particular,  for such $f$ we have
\[
\as X = \as f(X) \qquad \text{and} \qquad \dim_\theta X = \dim_\theta f(X)
\]
for all $\theta \in (0,1)$.  This was proved for the Assouad spectrum in \cite{Spectraa} and we prove it for the intermediate dimensions here.

\begin{lemma}
For any bounded set $X \subseteq \mathbb{R}^d$ and bi-Lipschitz map $f:X \to \mathbb{R}^d$, we have  $\dim_\theta X = \dim_\theta f(X)$ for all $\theta \in (0,1)$.
\end{lemma}

\begin{proof}
Let $s> \dim_\theta X$ and $\varepsilon>0$.  It follows that there exists a cover $\{U_i\}$ of $X$ with $|U_i| \leq |U_j|^\theta$ for all $i,j$ such that $ \sum_i |U_i|^s < \varepsilon$. It follows that $\{f(U_i)\}$ is a cover of $f(X)$ and that  $|f(U_i)| \leq C|U_i|  \leq C |U_j|^\theta \leq C^{1+\theta} |f(U_j)|^\theta $ for all $i,j$, where $C$ is the  constant from \eqref{bilipz}. Let $\delta = \inf_j |f(U_j)|$.  For all $i$ such that $\delta^\theta < |f(U_i)|  \leq C^{1+\theta} \delta^\theta$, cover the  set $f(U_i)$ with balls of diameter $\delta^\theta$ and replace the covering set $f(U_i)$ by these balls.  Note that we can always do this with fewer than $c_dC^{d(1+\theta)}$ balls where $c_d \geq 1$ is a constant depending only on $d$.  This yields an allowable cover $\{V_l\}$ of $f(X)$ and we have
\[
\sum_l |V_l|^s \leq c_dC^{d(1+\theta)} \sum_i C^s|U_i|^s  \leq c_dC^{d(1+\theta)+s}\varepsilon
\]
which proves $\dim_\theta f(X) \leq \dim_\theta X$ by letting $s \to\dim_\theta X$.  The reverse inequality follows by replacing $f$ by  $f^{-1}$ in the above. 
\end{proof}

An immediate consequence of the bi-Lipschitz invariance of the Assouad spectrum is that if $F_1$ and $F_2$ are Bedford-McMullen carpets associated with $m_1 \times n_1$ and $m_2 \times n_2$ grids, respectively, and there exists a bi-Lipschitz map between $F_1$ and $F_2$, then
\[
\frac{\log m_1}{\log n_1} = \frac{\log m_2}{\log n_2}.
\]
This is because this ratio corresponds to the phase transition in the  spectrum, and is therefore a bi-Lipschitz invariant.  This is not at all surprising, but serves as a simple example of the spectrum yielding applications which are not immediate when considering the dimensions in isolation.  Classification of self-affine sets up to bi-Lipschitz equivalence is an interesting problem, see \cite{miao}.

Bi-H\"older maps are a natural generalisation of bi-Lipschitz maps where more distortion is allowed.  We say an injective map $f: X \to\mathbb{R}^d$ is $(\alpha, \beta)$-H\"older, or bi-H\"older,  for $0<\alpha \leq 1 \leq \beta < \infty$  if there exists a constant $C \geq 1$ such that for all distinct $x,y \in X$
\[
C^{-1} |x-y|^\beta \leq  |f(x)-f(y)| \leq C |x-y|^\alpha.
\]
We note that being  $(1,1)$-H\"older is the same as being bi-Lipschitz.  Dimensions are typically not preserved under bi-H\"older maps, but one can often control the distortion.  For example, if $\dim$ is the Hausdorff, or upper or lower box dimension, and $f$ is $(\alpha, \beta)$-H\"older, then
\begin{equation} \label{control}
\frac{\dim X}{\beta} \leq \dim f(X)  \leq \frac{\dim X}{\alpha},
\end{equation}
see \cite[Proposition 3.3]{falconer}. Notably, the Assouad dimension does not satisfy such bounds, see \cite[Proposition 1.2]{LuXi}.  The Assouad spectrum, which is inherently more regular than the Assouad dimension, \emph{can} be controlled in this context but the control is more complicated than \eqref{control}.  The following lemma is adapted from \cite[Proposition 4.7]{Spectraa}.

\begin{lemma}  \label{Holderthm}
Suppose  $f : X\rightarrow \mathbb{R}^d$ is $(\alpha, \beta)$-H\"older.  Then,  for all $\theta \in (0,1)$, 
\[
\frac{1-\beta \theta/\alpha}{\beta(1-\theta)} \,  \dim_\mathrm{A}^{\beta \theta/\alpha} X \  \leq \  \dim_\mathrm{A}^{\theta} f(X) \  \leq \  \frac{1-\alpha \theta/\beta}{\alpha(1-\theta)} \, \dim_\mathrm{A}^{\alpha \theta/\beta} X
\]
where $\dim_\mathrm{A}^{\beta \theta/\alpha} X$ is taken to equal 0 if $\beta \theta/\alpha \geq 1$.
\end{lemma}

In order to motivate this result, we consider the \emph{winding problem}.   Given $p \geq 1$, let 
\[
\mathcal{S}_p= \{ x^{-p} \exp(ix) : 1<x<\infty\}
\]
which is a polynomially winding spiral with focal point at the origin. The winding problem concerns quantifying how little distortion is required to map $(0,1)$ onto $\mathcal{S}_p$.  For example, if $x^{-p}$ is replaced by $ e^{-cx}$ for some $c>0$, it is possible to map $(0,1)$ onto the corresponding spiral via a bi-Lipschitz map, see  \cite{unwindspirals}.  However, this is not possible for the spirals $\mathcal{S}_p$, see  \cite{spirals}. Therefore, it is natural to consider bi-H\"older winding functions, and attempt to optimise the H\"older exponents.  

Here there is a possible application of dimension theory: if the dimensions of $\mathcal{S}_p$ can be computed, and strictly exceed $1$, then \eqref{control} (or similar) will  directly lead to bounds on the possible H\"older exponents for winding functions $f: (0,1) \to \mathcal{S}_p$.  However, since $\mathcal{S}_p$ can be broken up into a countable collection of bi-Lipschitz curves, it follows that  $\hd \mathcal{S}_p = 1$.  Moreover, it was proved in \cite{fraserwinding} that $\bd \mathcal{S}_p = 1$.  This does not follow from  the countable decomposition since box dimension is not countably stable. Therefore, neither the Hausdorff nor box dimensions give any information on the H\"older exponents.   It was proved in \cite{fraserwinding} that $\ad \mathcal{S}_p = 2$, but despite this being strictly greater than $\ad (0,1) = 1$, we also get no information from the Assouad dimension since the change in dimension  cannot be controlled by the H\"older exponents. 

 It was proved in \cite{fraserwinding} that 
\[
\as \mathcal{S}_p =  1+ \frac{\theta}{p(1-\theta)} 
\]
for $0<\theta <\frac{p}{1+p}$, and 
\[
\as \mathcal{S}_p = 2
\]
for $\frac{p}{1+p}\leq \theta < 1$, see Figure \ref{fig:spiralspec}. Therefore, since we do have some control on how the Assouad spectrum distorts under bi-H\"older maps, this dimension formula \emph{does} yield non-trivial information.  Specifically, we get that if $f : (0,1) \to \mathcal{S}_p$ is an $(\alpha, \beta)$-H\"{o}lder map, then
\begin{equation} \label{windbound}
\alpha \leq \frac{p\beta +\beta}{ p+2\beta}.
\end{equation}
This follows by applying the first inequality in Lemma \ref{Holderthm} to $f^{-1}$ for $\theta = \alpha p/(\beta p+\beta)$. In particular, if $\beta = 1$, then $\alpha \leq \frac{p+1}{p+2}<1$, which is a stronger, quantitative, analogue  of the fact that $(0,1)$ cannot be mapped to $\mathcal{S}_p$ via a bi-Lipschitz map. 

It turns out that the bounds \eqref{windbound} are \emph{not} sharp.  The sharp relationship between $\alpha$ and $\beta$ is given by
\[
\alpha  \leq \frac{p\beta}{p+\beta},
\]
see \cite{fraserwinding} and Figure \ref{fig:holderbounds}.  We note the amusing resemblance of this relationship to that of \emph{Sobolev conjugates}.  Recall the \emph{Sobolev embedding theorem} which says that,  for $1 \leq p<d$, one has 
\[
W^{1,p}(\mathbb{R}^d) \subset L^{q}(\mathbb{R}^d)
\]
where $q$ is defined by
\[
p = \frac{dq}{d+q},
\]
that is, $q$ is the Sobolev conjugate of $p$.

\begin{figure}[H]
\centering
\includegraphics[width=0.87\textwidth]{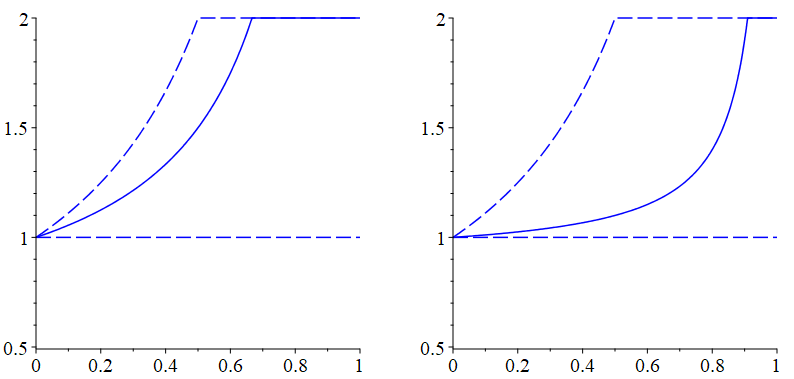}
\caption{Plots of $ \as\mathcal{S}_p$  (solid blue) as a function of $\theta$.  On the left $p=2$ and on the right $p=10$.  For reference, the general upper and lower bounds for the Assouad spectrum from  Lemma \ref{standardbounds}  are shown as dashed blue lines.  }
\label{fig:spiralspec}       % Give a unique label
\end{figure}

\begin{figure}[H]
\centering
\includegraphics[width=0.87\textwidth]{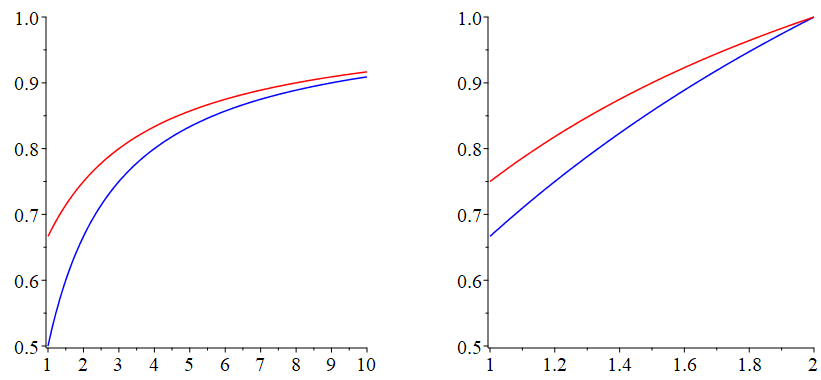}
\caption{Left: a plot of the upper bounds for $\alpha$ as a function of $p$ where $\beta = 1$ is fixed.  The sharp upper bound is shown in blue and the upper bound given by the Assouad spectrum is shown in red.  Right: a plot of the upper bounds for $\alpha$ as a function of $\beta$ where $p=2$ is fixed. The sharp upper bound is shown in blue and the upper bound given by the Assouad spectrum is shown in red.  }
\label{fig:holderbounds}       % Give a unique label
\end{figure}

A further application of the Assouad spectrum in this context is that $\as \mathcal{S}_p$, distinguishes spirals with different winding rates $p$.  Note that this is \emph{not} achieved by the Hausdorff, box, or Assouad dimensions, since these (somewhat surprisingly) do not depend on $p$.  In particular, the Assouad spectrum shows that $\mathcal{S}_p$ and $\mathcal{S}_{q}$ are not bi-Lipschitz equivalent for $p \neq q$.

\section{Further remarks}
\label{sec:4}

We note that the Assouad spectrum of the spirals considered in the previous section exhibits a single phase transition at $\frac{p}{p+1}$.  Similar to the self-affine carpets, it is easy to see that this phase transition occurs strictly to the right of the phase transition in the general upper bound, provided $p>1$, and therefore the general upper bound is not realised by these spirals.  This gives rise to a similar form for the spectrum of the carpets and the spectrum of the spirals.   We observe that this similarity goes a little deeper.  In fact, in both cases we have the formula
\begin{equation} \label{funformula}
\as E =  \min \left\{ \dim_\mathrm{B} E + \frac{( 1- \rho)\theta }{(1-\theta)\rho}  \left( \dim_\mathrm{A} E - \dim_\mathrm{B} E\right) , \ \dim_\mathrm{A} E \right\},
\end{equation}
where $\rho$ is a constant which holds particular geometric significance for the object  $E$.  Specifically, for carpets $\rho  = \frac{\log m}{\log n}$, and for  spirals $\rho  = \frac{p}{p+1}$.  Also note that $\rho $ is the   value of $\theta$ at which the unique phase transition occurs.  In both cases $\rho $ captures some fundamental scaling property of the set.  For carpets, the  $k$th level rectangles in the standard construction of $F$  are of size $m^{-k} \times n^{-k}$ and therefore $\rho $ is the ``logarithmic eccentricity''.  For spirals, the \emph{$k$th  revolution}, given by
\[
\{ x^{-p} \exp(ix) : 1+2\pi(k-1)<x\leq 1+2\pi k\},
\]
has diameter comparable to $k^{-p}$, while the distance between the end points (or, outer radius minus inner radius) is comparable to $k^{-(p+1)}$.  These are fundamental measurements considered in the winding problem, see \cite{fraserwinding}, and measure how big the $k$th revolution is and how tightly it  is wound, respectively.  Again the ``logarithmic eccentricity'' is
\[
\frac{\log \left(k^{-p} \right)}{\log \left( k^{-(p+1)} \right) }= \frac{p}{p+1} = \rho.
\]
We wonder if this is a coincidence, or whether it is reflective of a more general phenomenon.  It would be interesting to identify other natural classes of set for which this formula holds for a particular choice of ``fundamental ratio'' $\rho$. Finally, we note that the Assouad spectrum does \emph{not} generally satisfy an equation of the form \eqref{funformula}, see \cite{canadian,Spectraa,Spectrab}.

\section*{Acknowledgements}
The author thanks Stuart Burrell, Kenneth Falconer,   Kathryn Hare, Kevin Hare, Antti K\"aenm\"aki, Tom Kempton, Sascha Troscheit,  and Han Yu for many interesting discussions relating to dimension interpolation.  He was financially supported in part by the  EPSRC Standard Grant EP/R015104/1.  He is also grateful to the Leverhulme Trust for funding his project \emph{New Perspectives in the dimension theory of fractals} (2019-2023), which is largely focused on the concept of dimension interpolation. Finally, he thanks Stuart Burrell and Kenneth Falconer for making several helpful comments on an earlier version of this article.


\begin{thebibliography}{99}%

\bibitem{bedford}
Bedford, T.:  Crinkly curves, Markov partitions and box dimensions in self-similar sets.  PhD thesis, University of Warwick  (1984)

\bibitem{bishopperes}
Bishop, C. J., Peres, Y.: Fractals in Probability and Analysis. Cambridge studies in advanced mathematics, \textbf{162} (2017)

\bibitem{chen1}
Chen, H., Wu, M., Chang, Y.: Lower Assouad type dimensions of uniformly perfect sets in doubling metric spaces.
preprint, available at  https://arxiv.org/abs/1807.11629

\bibitem{chen2}
Chen, H.: Assouad dimensions and spectra of Moran cut-out sets.
Chaos Sol. Fract. {\bf 119},  310--317 (2019)

\bibitem{techniques}
Falconer, K. J.: Techniques in Fractal Geometry, John Wiley (1997)

\bibitem{falconer}
Falconer, K. J.:  Fractal Geometry: Mathematical Foundations and Applications.  John Wiley, 3rd. ed. (2014)

\bibitem{FalconerFraserKempton}
Falconer, K. J., Fraser, J. M., Kempton,  T.: Intermediate dimensions. preprint available at: https://arxiv.org/abs/1811.06493

\bibitem{spirals}
Fish, A., Paunescu, L.: Unwinding spirals.  Methods and Applications of Analysis, to appear,  available at: http://arxiv.org/abs/1603.03145 

\bibitem{fraserassouad}
Fraser, J. M.:   Assouad type dimensions and homogeneity of fractals.
Trans. Amer. Math. Soc. {\bf 366}, 6687--6733 (2014)

\bibitem{fraserwinding}
Fraser, J. M.:  On H\"older solutions to the spiral winding problem. preprint available at: https://arxiv.org/abs/1905.07563







\bibitem{canadian}
Fraser, J. M., Hare, K. E.,   Hare, K. G.,  Troscheit, S.,  Yu, H.: The Assouad spectrum and the quasi-Assouad dimension: a tale of two spectra. Ann. Acad. Sci. Fenn. Math.  {\bf 44},   379--387 (2019)

\bibitem {fraserrob}
Fraser, J. M., Henderson, A. M., Olson, E. J., Robinson, J. C.: On the Assouad dimension of self-similar sets with overlaps. 
 Adv. Math. {\bf 273}, 188--214 (2015)

\bibitem{frasermiaotro}
Fraser, J. M., Miao, J.J.,Troscheit, S.:   The Assouad dimension of randomly generated fractals.
Ergodic Th. Dyn. Syst. {\bf 38},  982--1011 (2018)

\bibitem{frasersascha}
Fraser, J. M., Troscheit, S.: The Assouad spectrum of random self-affine carpets. preprint,  available at: https://arxiv.org/abs/1805.04643

\bibitem{Spectraa}
Fraser, J. M., Yu, H.: New dimension spectra: finer information on scaling and homogeneity. Adv. Math. {\bf 329}, 273--328 (2018)

\bibitem{Spectrab}
Fraser, J. M., Yu, H.:  Assouad type spectra for some fractal families. Indiana Univ.~Math.~J. {\bf 67}, 2005--2043 (2018)





\bibitem{Hare2}
Garc\'{i}a, I., Hare, K. E., Mendivil, F.: Intermediate Assouad-like dimensions. preprint,  available at  https://arxiv.org/abs/1903.07155

\bibitem{Hare3}
Garc\'{i}a, I., Hare, K. E., Mendivil, F.: Almost sure Assouad-like Dimensions of Complementary sets. preprint,  available at  https://arxiv.org/abs/1903.07800


\bibitem{Hare4}
 Hare, K. E., Troscheit, S.: Lower Assouad Dimension of Measures and Regularity.
preprint,  available at  https://arxiv.org/abs/1812.05573


\bibitem{unwindspirals}
Katznelson, Y., Nag, S., Sullivan, D.: On conformal welding homeomorphisms associated to Jordan curves. Ann. Acad. Sci. Fenn. Math. {\bf 15}, 293--306 (1990)

\bibitem{miao}
Li, B.,  Li, W.,   Miao, J.J.: Lipschitz equivalence of McMullen sets. Fractals {\bf 21} (2013)

\bibitem{LuXi}
L\"u, F., Xi, L.: Quasi-Assouad dimension of fractals. J. Fractal Geom. {\bf 3}, 187--215 (2016)

\bibitem{mackay}
Mackay, J. M.:
Assouad dimension of self-affine carpets. Conform. Geom. Dyn. {\bf 15}, 177--187 (2011)

\bibitem{mattila}
Mattila, P.: Geometry of sets and measures in Euclidean spaces. Cambridge studies in advanced mathematics, \textbf{44} (1995)

\bibitem{mcmullen}
McMullen, C. T.:
 The Hausdorff dimension of general Sierpi\'nski carpets. Nagoya Math. J. {\bf 96},  1--9 (1984)

\bibitem{robinson}
Robinson, J. C.: Dimensions, Embeddings, and Attractors. Cambridge University Press (2011)


\bibitem{pablo}
P. Shmerkin.
 On Furstenberg's intersection conjecture, self-similar measures, and the $L^q$ norms of convolutions.
 Ann.  Math. {\bf 189}, 319--391 (2019)  

\bibitem{Troscheit18}
Troscheit, S.: The quasi-Assouad dimension of stochastically self-similar sets.
Proc. Roy. Soc. Edinburgh Sect. A, to appear,  available at: https://arxiv.org/abs/1709.02519

\bibitem{sascha}
Troscheit, S.: Assouad spectrum thresholds for some random constructions. 
preprint available at: https://arxiv.org/abs/1906.02555 




\bibitem{hanphd}
Yu, H.:  Assouad type dimensions and dimension spectra for some fractal families. PhD thesis, The University of St Andrews (2019)

\end{thebibliography}
\end{document}